\documentclass[12pt,reqno]{amsart}
\usepackage{geometry}                		
\usepackage{graphicx}				
\usepackage{amssymb}
\usepackage{epstopdf}
\usepackage{datetime}
\usepackage{currfile} 
\allowdisplaybreaks

\usepackage{color}

\newtheorem{thm}{Theorem}[section]

\newtheorem{lem}[thm]{Lemma}
\newtheorem{prop}[thm]{Proposition}
\theoremstyle{definition}

\theoremstyle{remark}
\newtheorem{rem}{Remark}[section]

\numberwithin{equation}{section}
						
\numberwithin{equation}{section}

\newcommand{\red}[1]{\textcolor{black}{#1}}

\begin{document}

\title[]
      {Selection dynamics for deep neural networks}
       \author{Hailiang  Liu and Peter Markowich }
\address{Iowa State University, Mathematics Department, Ames, IA 50011} \email{hliu@iastate.edu}
\address{Computer, Electrical, Mathematical Sciences and Engineering Division, King Abdullah University of Science and Technology (KAUST), Thuwal, Saudi Arabia.} \email{Peter.Markowich@kaust.edu.sa}



\begin{abstract} 
This paper presents a partial differential equation framework for deep residual neural networks and for the associated learning problem. This is done by carrying out the continuum limits of neural networks with respect to width and depth.  We  study the wellposedness, the large time solution behavior, and the characterization of the steady states of the forward problem. Several useful time-uniform  estimates and stability/instability conditions  are presented. We state and prove optimality conditions for the inverse deep learning problem, using standard variational calculus, the Hamilton-Jacobi-Bellmann equation and the Pontryagin maximum principle.
This serves to establish a mathematical foundation for investigating the algorithmic and theoretical connections between neural networks, PDE theory, variational analysis, optimal control, and deep learning.
\end{abstract}
\maketitle


\section{Introduction}
Deep learning is machine learning using neural networks with many hidden layers, and it \cite{Be09, LBH15, GBC16} has become a primary tool in a wide variety of  practical learning tasks, such as image classification, speech recognition,  driverless cars, or game intelligence.   As such, there is a pressing need to provide a solid mathematical framework to analyze various aspects of deep neural networks. 

Deep Neural Networks (DNN) have been successful in supervised learning, particularly when the relationship between the data and the labels is highly nonlinear. Their depths allow DNNs to express complex data-label relationships since each layer nonlinearly transforms the features and therefore effectively filters the information content.

Linear algebra was appropriate in the age of shallow networks, but is inadequate to explain why deep networks perform better than shallow networks. 
The continuum limit is an effective method for modeling complex discrete structures to facilitate their interpretability. 
The depth continuum limit made a breakthrough by introducing a dynamical system viewpoint and going beyond what  discrete networks can actually do.

Most prior works on the dynamical systems viewpoint of deep learning have focused on algorithm design, architecture improvement using  ODEs to model residual neural networks. However, the ODE description does not  reveal any structure for hidden nodes with respect to width. To fill in this gap, we propose a simple PDE model for DNNs that  represents the continuum limits of deep neural networks with respect to two directions: width and depth. 
One main advantage of the PDE model over the ODE model in \cite{HR17} is its ability to capture the intrinsic selection dynamics among hidden units involved. 

Also, and maybe most importantly, the forward and backward PDE problems can be discretized by numerical methods which lead to different network architectures (with respect to depth and width) than the empirical explicit Euler scheme that was at the basis of the depth continuum process. This allows much more mathematical sophistication with possible rewards in stability, efficiency and speed gains compared to the simple layer-by-layer iteration technique for the forward and backward problems. Moreover, in many applications it is sensible to limit the parameter space by requiring the learning parameters to remain in bounded sets. Then the minimization procedure can be replaced by a control theory approach leading to completely coupled forward-backward PDEs, where the coupling occurs through the optimal controls. Numerical techniques for control theory problems and mean field games can then be used, replacing the classical variational methods (adapted gradient descent) often used for ML in neural networks. We remark that limiting the parameter space is often preferential to classical (Tikhonov-type) regularization procedures of the objective functional.

The main purpose of this paper is to focus on the study of the fundamental mathematical aspects of the PDE formulation. We seek to gain new insight into the dynamics of the forward propagation and the well-posedness of the learning problem, through a study of the PDE that represents the forward propagation dynamics. In this framework the study of the impact of the choice of the activation function on the network dynamics, stability and on the learning problem becomes very apparent and rather straightforward to analyze by classical PDE techniques.

We point out that  the link of deep learning to dynamical system and  optimal control  has attracted increasing attention \cite{CMHRBM18, CHHTB18, EHL19,  HR17,  LCTE18,  LH18, LS17, LZLD17,  SM17}.  An appealing   feature of this approach is that the compositional structure is explicitly taken into account in the time evolution of the dynamical systems, from which novel algorithms and network structures can be designed.  
\subsection{Discrete neural networks}
A neural network can be seen as a recursively defined function $\Phi$ on (a compact domain of) $\mathbb{R}^d$ into $\mathbb{R}^{N_L}$:
$$
 \Phi = L_L \circ F_{L-1} \cdots L_2\circ  F_1\circ  L_1 .
$$
Here $L_k$ is an affine linear map from $\mathbb{R}^{N_{k-1}}$ to $\mathbb{R}^{N_k}$:
$$
L_k(x) = a_k - B_k x,
$$
where $B_k$ are $N_k \times N_{k-1}$ matrices (network weights) and $a_k$ are $N_k$-vectors (network biases). Obviously we have used $N_0 =d$.  $F_k$ is a nonlinear mapping from $\mathbb{R}^{N_k}$ into itself,  given by 
$$
 F_k(y) = {\rm diag}(\sigma(y_1), \cdots, \sigma(y_{N_k})):=:\sigma(y),
$$
where $\sigma:\mathbb{R}\to \mathbb{R}$ is the so called network activation function.  
For details in the setup of neural network functions and their approximation qualities we refer to \cite{BGKP18}.

Residual neural networks are set up slightly differently,  inasmuch they add the current state to the activation term on each level:
$$
z_{l+1}=z_l +\sigma(a_l -B_l z_l). 
$$
\red{Clearly, all the dimensions $N_k$ are assumed to be equal here, i.e., $d=N_0=N_1=\cdots=N_{L-1}=N_L$.}  After rescaling $\sigma$ with an artificial layer width $\tau <<1$, the recursion can be seen as an explicit Euler step of the system of ordinary differential equations:
$$
\frac{d}{dt} z(t) = \sigma(a(t) - B(t)z(t)), t_{l-1}\leq t\leq t_l, 
$$
where $t_l=\tau l$, and $t>0$ corresponds here to an artificially introduced time-like variable representing the depth of the network. 
By now this is a fairly common procedure in DNNs, we refer to \cite{CMHRBM18, CHHTB18, EHL19, HR17, LS17, LZLD17, SM17} and references therein.

We remark that more complicated network evolutions have been considered in the literature \cite{He16a, He16b, LMS16, TSDL18, WGGH18}, we shall comment on a specific example later on in this work.

In practical applications the dimensions $N_0\cdots N_l$ vary significantly from one network layer to the next, so in order not to have a-priori dimensional restrictions it makes sense to pose the above ODE system on an infinitely dimensional space of continuously defined functions. This is the approach which we shall take in this paper. 

Obvious advantages arise. In the space-time continuous case we gain a lot of modeling freedom and highly developed PDE theory and numerics can be applied to analyze and compute geometric aspects of the problem like attractors, sharp fronts etc.  Also the associated inverse problem, namely to determine the weight and bias functions based on given data, can be rephrased easily as a classical optimization and/or control problem.

Our approach allows to apply the very well developed PDE analysis and numerical analysis technology to deep learning neural network problems. We remark that the purpose of the paper is to present a PDE framework for deep learning for neural networks, opening up  ML to sophisticated analytical and numerical techniques which go beyond the current state of the art in ML. For this purpose we focus on the PDE formulation and certain basic mathematical properties while practical examples and more involved mathematical issues will be the subject of subsequent works.

\subsection{Organization}
The paper is organized as follows. We discuss main ingredients of deep learning for the classification problem  and derive the PDE model for the forward propagation and  the optimal control formulation of deep learning in Sect. 2. 
 In Sect. 3, we  study the wellposedness, the large time solution behavior, and the characterization of the steady states for the forward problem. Several useful a priori estimates and stability/instability conditions  are presented.  Sect. 4 is devoted to the back propagation problem and to show how optimal control theory can be applied. We compute the gradient of the final network loss in terms of the network parameter functions, which involves solving both forward and backward problems.   We further develop a control theory based on the Pontryagin maximum principle (PMP) \cite{Po87},  which provides explicit necessary conditions for optimal controls. 
{We finally show that the value function solves an infinite-dimensional Hamilton--Jacobi--Bellman (HJB) partial differential equation. This dynamic programming approach provides the third way to characterize the optimal control parameters. }
 Hence in this work we establish the link between training deep residual neural networks and PDE parameter estimation. The relation provides a general framework for designing, analyzing and training CNNs. Finally, two numerical algorithms for the learning problem, one is gradient based  and another is PMP based,  are presented in Sect. 5.
 \subsection{Related work} 
The approximation properties of deep neural network models are fundamental in machine learning.  For shallow networks, there has been a long history of proving the so-called universal approximation theorem, going back to the 1980s \cite{Cy89, Ho89}. Such universal approximation theorems can also be proved for wide networks,  see \cite{Ba93} for a single layer with sufficient number of  hidden neurons, or deep networks,  see \cite{Lu17, EW18} for networks of finite width with sufficient number of layers.  A systematic study on the network approximation theory has been recently made available \cite{GPEB19}.

Continuous time recurrent networks have been known in 1980s like the one proposed by Almeida  \cite{Al87} and Pineda \cite{Pi87}, and analyzed by LeCun \cite{Le88}. Recently, the interpretation of residual networks by He et al. \cite{He16a} as approximate ODE solvers in \cite{WE17} spurred research in the use of ODEs to deep learning. Based on differential equations, there are studies on the continuum-in-depth  limit of neural networks \cite{LS17, SM17} and on designing network architectures for deep learning \cite{CMHRBM18, CHHTB18, HR17, LZLD17}.  Indeed many state-of-the-art deep network architectures, such as PolyNet \cite{ZLL17} and FractalNet \cite{LMS16}, can be considered as different discretizations of ODEs \cite{Lu18}. Theoretical justification can be found in \cite{MG18}. For models motivated by PDEs we refer the reader to \cite{RH18}.

The dynamical systems approach has also been explored in the direction of training algorithms based on the PMP and the method of successive approximations \cite{LCTE18, LH18}. The connection between back-propagation and optimal control of dynamical systems is known since the earlier works on control and deep learning \cite{AF13, Br75, Le88}. For a rigorous analysis on formulations based on ODEs with random data we refer to \cite{EHL19}.   

The present paper proposes a PDE model which represents a continuum limit of neural networks in both depth and width.  Instead of the analysis of algorithms or architectures, we focus on the mathematical aspects of the formulation itself and develop a wellposedness theory for the forward and backward problems, and further characterize the optimality conditions and value functions using both  ODE (PMP) and PDE (HJB) approaches. 

\section{Mathematical formulation } 
There are three main ingredients of deep learning for the classification problem: (i) forward propagation transforms the input features in a nonlinear way to filter their information;
(ii) Classification is described to predict the class label probabilities using the features at the final output layer (i.e., the output of the forward propagation); 
and (iii) the learning problem is formulated to estimate parameters of the forward propagation and classification to approximate the data-label relation.

\red{We start out by deriving the PDE representing the network architecture, i.e, the forward propagation. }

\subsection{\red{The `Thermodynamic' limit}}
We shall now make the limit of infinite depth and infinite width of residual networks precise in analogy to multi-particle physics we refer to it as the thermodynamic limit process:

At first consider a network of  fixed width $N$ and large depth $L\gg 1$, i.e.,
\begin{align}\label{zz+}
z_{l+1}=z_l +\tau \sigma(a_l -B_l z_l), \quad l=0, 1, \cdots, L,
\end{align}
where $z_l, a_l \in \mathbb{R}^N$, $B_l$ are $N\times N$ matrices, $\sigma$ is the given activation function and $\tau>0$ is the 
(artificially chosen) layer width. We set $T=L\tau$. Note that the rescaling   of the activation function from $\sigma$ to $\tau \sigma$ makes sense since iteration overflow and instability may occur when the iteration is run for $L$ large. That is why we think of $\tau$ as a small `time step' parameter and of $\sigma$ as an $O(1)$-function.  To allow for more general and more interesting results we let the coefficients $a_{l}, B_l$ depend on $\tau$, we denote the $\tau-$dependence by a superscript, i.e., $a_{l}^\tau, B_l^\tau$  and $z_l^\tau$. We now define the piecewise linear functions 
$a^\tau=a^\tau(t)$ and $B^\tau=B^\tau(t)$ interpolating $a_l^\tau$ and, respectively, $B_l^\tau$ so that 
$$
a^\tau(t_l)=a_l^\tau, \; B^\tau(t_l)=B_l^\tau, \; l=0, \cdots, L. 
$$
Then we define the function $z^\tau:=z^\tau(t)$ on $[0, T]$ through the extended recursion 
\begin{subequations}\label{2ab}
\begin{align}
z^\tau(t+\tau)=z^\tau(t) +\tau \sigma(a^\tau(t) -B^\tau(t) z^\tau(t))
\end{align}
for $0\leq t\leq T-\tau$. We prescribe an initial datum for the recursion,
\begin{align}
z^\tau(t)=z_0, \quad \text{for}\; 0\leq t\leq \tau
\end{align}
\end{subequations}
with $z_0\in \mathbb{R}^N$ given. Note that $z^\tau(t_l)=z_l^\tau$ holds. We assume here that $\sigma$ grows almost linearly:
$$
(A_1) \qquad \exists C_1,  C_2>0\;\; \text{such that }\; |\sigma(s)|\leq C_1+C_2s \; \text{for} \; s\in \mathbb{R},
$$
and that $a^\tau, B^\tau$ are bounded independently of $\tau$ in $L^\infty$ as $\tau \to 0$:
$$
(A_2) \qquad \exists C_3\;\; \text{such that }\; \\
\|a^\tau\|_{L^\infty((0, T); \mathbb{R}^N)}+\|B^\tau\|_{L^\infty((0, T); \mathbb{R}^{N\times N})} \leq C_3 \quad \text{for} \; \forall \tau \; \text{small}.
$$
The following convergence result holds.
\red{\begin{thm} \label{thm2.1} Let $(A_1), (A_2)$ hold and assume that 
there are functions $a=a(t), B=B(t)$ such that 
\begin{align*}
& a^\tau(t) \overset{\tau \to 0}{\longrightarrow}  a(t) \quad p.w. a.e. \; \text{on} \; (0, T),\\
& B^\tau(t) \overset{\tau \to 0}{\longrightarrow}   B(t) \quad p.w. a.e. \; \text{on} \; (0, T).
\end{align*}
Then $z^\tau \overset{\tau \to 0}{\longrightarrow}  z$ as $\tau \to 0$ in $C([0, T]; \mathbb{R}^N)$, where 
$z=z(t)$ solves the IVP on $\mathbb{R}^N$: 
\begin{subequations}\label{zz}
\begin{align}
& \dot z=\sigma(a-Bz), \; 0 \leq t \leq T,\\
& z(0)=z_0.
\end{align}
\end{subequations}
\end{thm}
}
\begin{proof} From $(A_1)$ and $(A_2)$ we immediately 
conclude that the solution $z^\tau$ of (\ref{2ab}) is uniformly bounded as $\tau \to 0$. 
Therefore (using recursion) also the difference quotient 
$$
\frac{z^\tau(t+\tau)-z^\tau(t)}{\tau}
$$
is uniformly bounded as $\tau \to 0$ and the Arzela-Ascoli Theorem  
implies that -- upon extraction of a subsequence as $\tau \to 0$ -- 
there exists a function $z\in C([0, T])$ such that 
\begin{align*}
& z^\tau \to z \quad  \text{in} \; C([0, T]),\\
&\frac{z^\tau(t+\tau)-z^\tau(t)}{\tau}  \to \dot z(t)  \quad \text{in} \; \mathcal{D}'[0, T).
\end{align*}
The result then follows by integrating 
$$
\frac{z^\tau(t+\tau)-z^\tau(t)}{\tau}=\sigma(a^\tau(t)-B^\tau(t)z^\tau(t)),
$$
against a function $\phi \in \mathcal{D}[0, T)$ and passing to the limit $\tau \to 0$ (using the Lebesque dominated convergence Theorem on the right). This defines the `large depth' limit process of the residual neural networks.
\end{proof}
Note that the assumptions $(A_1), (A_2)$ and in particular the convergence assumption of Theorem 1 require the coefficients $a_l, B_l$ to depend on the depth of the networks in an appropriate way. 

In order to understand the `infinite width' limit $N\to \infty$ we set out by choosing a dimension $n\in \mathbb{N}$ and an open Jordan set $Y\subset \mathbb{R}^n$. Let $\{Y_1, \cdots, Y_N\}$ be a disjoint partition of $Y$, where each $Y_k$ is an open Jordan set with 
$ \overline{Y}=\cup_{k=1}^N  \overline{Y_k}$, such that 
$$
\int_Y f(v)dv=\sum_{k=1}^N\int_{Y_k} f(v)dv
$$
for any Lebesque integrable function $f: Y \to \mathbb{R}$. 
Now we denote the components of $a=a(t)\in \mathbb{R}^N$, $B=B(t)\in \mathbb{R}^{N\times N}$ by 
\begin{align*}
a(t)=(a^1(t), \cdots, a^N(t)),\quad B(t)=(B^{ij}(t))_{i, j=1, \cdots, N},
\end{align*}
and define the function 
\begin{subequations}\label{3ab}
\begin{align}
& a(y, t)=a^k(t) \quad \text{if}\; y\in Y_k,\\
& b(y, w, t)=\frac{1}{|Y_j|}B^{ij}(t) \quad  \text{if} \in y\in Y_i, w\in Y_j.
\end{align}
\end{subequations}
Clearly $a: Y\times [0, T] \to \mathbb{R}$, $b: Y\times Y \times [0, T] \to \mathbb{R}$ are defined a.e. on their respective domains. 

Consider now the IVP 
\begin{subequations}\label{4ab}
\begin{align}
& f_t(y, t)=\sigma\left(a(y, t)- \int_Y b(y, w, t)f(w, t)dw\right),\\
& f(y, t=0)=f_I(y), \quad y\in Y,
\end{align}
\end{subequations}
where $f_I$ is given by 
$$
f_I(y)=z_0^k, \quad \text{if}\; y\in Y_k. 
$$
Here $z_0=(z^1_0, \cdots, z^N_0)$ is the initial datum for the ODE system of Theorem \ref{thm2.1}. Clearly the (unique) solution of the IVP (\ref{4ab})  is the function 
$$
f(y, t)=z^k(t), \quad \text{if}\; y\in Y_k, 
$$
where $z(t)=(z^1(t), \cdots, z^N(t))$ is the solution of the IVP for the ODE system  (\ref{zz}). 

This clearly shows that the function values of the solution $f$ represent the residual networks while the arguments $y \in Y$  can be interpreted as their labels. 

To understand the limit process $N\to \infty$ and to eliminate technicalities we assume that $Y$ is the unit cube in $\mathbb{R}^d$, i.e., $Y=(0, 1)^d$. We assume $N=M^d$, consider $M\to \infty$ and denote $h=\frac{1}{M}$. Motivated by (\ref{3ab}b) we let the coefficients of the ODE system depend on $h$ and rescale the matrix operator. We consider 
 \begin{subequations}\label{5ab}
\begin{align}
& \dot z_h =\sigma(a_h- h^d B_h z_h),\; t\geq 0,\\
& z_h(t=0)=z_0,
\end{align}
\end{subequations}
 where 
 \begin{align*}
&  z_h =(z_h^1(t), \cdots, z_h^{M^d}(t)), \quad z_0=(z_0^1, \cdots, z_0^{M^d}),\\
&  a_h =(a_h^1, \cdots, a_h^{M^d}), \quad B_h=(b_h^{ij})_{i, j=1,\cdots, M^d}).
\end{align*}
We now partition $Y$ into cube cells $Y_j$ of width $h$, numbered arbitrarily: $Y=\cup_{j=1}^{M^d}Y_j$ and define the function
 \begin{align*}
& a_h(y, t)  =a_h^j(t)  \; \text{if} \; y\in Y_j, \\
& b_h(y, w, t)=b_h^{ij}(t) \quad \text{if}\; y\in Y_i, \; w\in Y_j,\\
&  f_h(y, t) =z_h^j(t), \; \text{if} \; y\in Y_j, \\
&    z_{0, h}(y) =z_0^j, \; \text{if} \; y\in Y_j.
\end{align*} 
Then (\ref{5ab}) is equivalent to 
\begin{subequations}\label{6ab}
\begin{align}
& \partial_t f_h(y, t)=\sigma(a_h(y, t)- \int_Y b_h(y, w, t)f_h(w, t)dw), \quad t>0, \; y\in Y,\\
& f_h(y, t=0)=z_{0, h}(y), \quad y\in Y.
\end{align}
\end{subequations}
We make the following assumption: 
$$
(A_3) \qquad \exists C_4>0 \; \text{such that} \; \; 
\|a_h\|_{L^\infty(Y\times (0, T))} +\|b_h\|_{L^\infty(Y\times Y \times (0, T))} 
+\|z_{0, h}\|_{L^\infty(Y)} \leq C_4
$$
for all $h$ sufficiently small. 

The following result is easy to show. 
\red{\begin{thm}Let $\sigma$ satisfy $(A_1)$, let $(A_3)$ hold and assume: 
\begin{align*}
& a_h(y, t)\overset{h \to 0}{\longrightarrow} a(y, t) \; p.w.a.e. \; \text{in}\; Y\times (0, T), \\
& b_h(y, w, t) \overset{h \to 0}{\longrightarrow}  b(y, w, t) \; p.w.a.e. \; \text{in}\; Y\times Y \times (0, T), \\
& z_{0, h}(y) \overset{h \to 0}{\longrightarrow}  f_I(y) \; p.w.a.e. \; \text{in}\; Y. 
\end{align*}
Then the solution $f_h$ of (\ref{6ab}) satisfies 
$$
f_h \stackrel{h \to 0}{\longrightarrow}  f \quad \text{in} \; \mathcal{D}'(Y\times [0, T)),
$$
where $f$ solves 
\begin{align*}
& f_t(y, t)=\sigma\left(a(y, t)- \int_Y b(y, w, t)f(w, t)dw\right), \; t>0, \; y\in Y,\\
& f(y, t=0)=f_I(y), \quad y\in Y.
\end{align*}
\end{thm}}
The result carries over to much more general domains $Y$ albeit with additional proof technicalities. We also remark that the label set $Y$ and its dimension $n$ can be chosen freely to contribute to the network architecture. 

\subsection{The forward problem}
The forward problem amounts to modeling and simulating the propagation of data. With complex and huge sets of data, fast and accurate forward modeling is a significant step in deep learning.  It should be noted that typically the more parameters the model has, the less well-posed the inverse problem is. 

We first formulate a forward PDE to model the data propagation using residual neural networks \cite{He16a}.  
Let $y\in Y$ denote the neuron identifier variable. Here we assume that $Y$ is a domain in $\mathbb{R}^n$. 
In order to construct a PDE-type model to describe the forward propagation in deep learning, we introduce an artificial time $t\in [0, T]$. The depth of the network is represented by the final time $T$. 
Let $f(y, t)$ be a function describing the residual neural network at time $t$ with neuron identifier $y$, its propagation is governed by the following PDE (of integro--differential type):
\begin{align}\label{main}
\partial_t f(y, t)=\sigma \left( a(y, t)- \int_{z\in Y} b(y, z, t)f(z, t)dz\right),
\end{align}
where $\sigma$ is the nonlinear activation function. Here $b=b(y, z, t)$ is the selection weight function, and $a=a(y, t)$  is the  bias function.  The input learning data set $f(y, t=0)=f_I(y)$ then serves as  the initial data for the above differential equation. One of our objectives in this work is to highlight the relation of the learning problem to this PDE model.

The activation function is typically (piecewise) smooth and monotonically non-decreasing. As commonly used examples, we consider the arctan, the hyperbolic tangent, the sigmoid of form $\frac{1}{1+e^{-s}}$, and the Rectified Linear Unit (ReLU) activations given by 
$$
\sigma(s)=s^+,
$$ 
the positive part of $s$. Our results also apply to other choices such as the leaky ReLu defined by $\sigma(s)=\max\{0.1s, s\}$, and the Elu given by 
$$
\sigma(s)=\left\{ 
\begin{array}{ll}
s & s>0,  \\
\alpha(e^s-1) & s\leq 0.
\end{array}
\right.
$$
The performance of these activation functions varies on different tasks and data sets \cite{Ra17} and it typically requires a parameter to be turned. Thus, the ReLU remains one of the popular activation functions due to its simplicity and reliability \cite{GBB11, LBH15, NH10}. 

The network output function  at final time $T$ is given by 
\begin{align}\label{oty}
O_T(y) := \int W(y, z)f(z, T)dz +\mu(y),
\end{align}
where $W$ and $\mu$ are weight and bias functions to be determined later.  

\subsection{The learning problem} \label{sec2.2}
In order to complete the learning problem, we need to define a prediction function by 
$$
C^{\rm pre}=h\left(O_T(y)\right).
$$
One of the popular choices for $h$ is the logistic regression function, 
$$
h(\xi)=e^\xi/(1+e^{\xi}).
$$
The goal of the learning problem is to estimate the parameters of the forward propagation (i.e., $a$ and $b$ and the classifier $W$ and $\mu$) from an  observed label function $C=C(y)$, so that the DNN accurately approximates the data-label relation for the training data and generalizes to new unlabeled data. 
The forward operator is highly nonlinear, and the learning problem most often does not fulfill Hadamard's postulate of well-posedness.

As we show below, the learning problem can be cast as a dynamic control problem, which provides new opportunities for applying theoretical and computational techniques from parameter estimation to deep learning problems. 

We phrase learning as an optimization problem 
\begin{subequations}\label{op}
\begin{align}
& \min J(C^{\rm pre}, C) \\
& \text{such that }\; \partial_t f(y, t)=\sigma\left(a(y, t)-\int_Y b(y, z, t)f(z, t)dz\right), \quad t\in (0, T],
\end{align}
\end{subequations}
where $J$ is a suitable choice of objective/loss function characterizing the difference between the synthetic data $C^{\rm pre}$ generated by the current (and inaccurate) model parameter $m=(a, b, W, \mu)$ and the observable true label $C$. This is a data-fitting approach, similar to many other inverse problems that are formulated as PDE-constrained optimization.

The optimization problem in (\ref{op}) is challenging for several reasons. Firstly, it is a high-dimensional non-convex optimization problem, and therefore one has to be content with local minima. Secondly, the computational costs per example are high, and the number of examples is large. Thirdly, very deep architectures are prone to problems such as vanishing and exploding gradients  that may occur when the discrete forward (or backward) propagation is unstable.

\subsection{The choice of the objective function and regularization} \label{sec2.3}
Typically the loss function $J$ is chosen to be convex in its first argument and measures the quality of the predicted class label probabilities.  A typical choice is 
$$
J(C^{\rm pre}, C)=\frac{1}{2}\int_Y|C^{\rm pre}(y)-C(y)|^2dy.
$$
For classification the cross entropy loss is often used to measure the model performance \cite{RVCPV04, MV14}.

To control noise and other undesirable effects occurring in inverse problems, one often adds a regularization term so that 
\begin{subequations}\label{op+}
\begin{align}
& \min J(C^{\rm pre}, C) +\lambda R(m),
\end{align}
\end{subequations}
where the  regularizer $R$ is a convex penalty functional,
and the parameter $\lambda > 0$ balances between minimizing the data fit and noise control. 
 Choosing an ``optimal" regularizer, $R$, and regularization parameter $\lambda$ is both crucial and nontrivial; see, e.g., \cite{Bi06, GBC16}.

\section{Wellposedness of the forward problem}  In order to identify useful structures of the deep learning problem, we first make some assumptions on the parameters with which stability of the forward problem can be studied. 

\subsection{A general existence result}
Most of our results will be obtained under the following:

{\bf Assumption 1}.  $Y$ is a domain in $\mathbb{R}^n$ and $0<T<\infty$. The propagation operator 
$$
\sigma(S[f]) \; \text{with} \; S[f]=a(y, t)- \int_{z\in Y} b(y, z, t)f(z, t)dz
$$
satisfies: 
\begin{itemize}
\item $\sigma$ is globally Lipschitz continuous with $\sigma(0)=0$ or $|Y|<\infty$. 
\item $a\in L^1((0, T), L^2(Y))$. 
\item $b\in L^1((0, T), L^2(Y\times Y))$ 
\end{itemize}
Note that in the course of this paper we assume that the integral operator is bounded on $L^2(Y)$ a.e.  in  $t$. This is clearly a restriction, but it represents the most common state of the art in network propagation applications. More general (differential or pseudo-differential) operators may be considered but this gives an entirely different flavor to the following analysis and is the topic of further work.

The above conditions are sufficient to prove the following theorem of existence and uniqueness by iteration. 
\begin{thm}We suppose that Assumption 1 holds. Then,  \\
1. for any initial  function $f_I \in L^2(Y)$ there exists a solution in $C([0, T]; L^2(Y))$ 
which solves (\ref{main}) with $f(0,\cdot)=f_I$. Furthermore,\\
2. for any two  solutions $f_1, f_2$ of (\ref{main}) in $C((0, T); L^2(Y))$, 
one has the following stability property:\\
\begin{align}\label{diff}
\forall t\in [0,T], \; \|f_1(t,\cdot)- f_2(t,\cdot)\|_{L^2(Y)} \leq e^{Lt}\|f_1(0,\cdot)-f_2(0,\cdot)\|_{L^2(Y)}, 
\end{align}
 for some $L>0$.  In particular, if $f_1(\cdot, 0)=f_2(\cdot, 0)$, then $f_1(\cdot, t)=f_2(\cdot, t)$ for all  $t>0$, so that uniqueness holds.
\end{thm}
\begin{proof} Existence follows from the recursive scheme
\begin{align*}
& f^0(y, t)=f_I(y),\\
& \partial_t f^{n+1} = \sigma(S[f^{n}]), \quad f^{n+1}(0) = f_I.
\end{align*}
For $f^n \in C((0,T); L^2(Y)):=Q$, $f^{n+1}$ is well-defined in the same space by 
$$
 f^{n+1}(t)=f_0+\int_0^t  \sigma(S[f^{n}])(\tau)d\tau,
$$
which in the $Q$ norm is bounded by 
\begin{align*}
&   \|f_I\|_{L^2} + \sup |\sigma'(\cdot)|  \left( \int_0^T (\|a(\cdot, \tau)\|_{L^2(Y)}d\tau  + 
\int_0^T \|b(\cdot, \cdot, \tau)\|_{L^2(Y\times Y)}d\tau \|f^n\|_{Q} +C_0T\right) \\
& \leq  \|f_I\|_{L^2}+C_0T + C_1 + C_2(T) \|f^n\|_{Q}
\end{align*}
where $C_0=|\sigma(0)||Y|$ if $\sigma(0)\not=0$, $C_1=\sup|\sigma'(\cdot)|  \|a\|_{L^1((0, T), L^2(Y))}$,  and 
$$
C_2(T) =\sup|\sigma'(\cdot)| \int_0^T \|b(\cdot, \cdot, \tau)\|_{L^2(Y\times Y)}d\tau.
$$
 Note that we used the well-known fact that the operator norm in $L^2$ of the integral operator with the kernel $b$ equals the norm of $b$ in $L^2$, that is $\|b\|_{L^2(Y\times Y)}$. By using the fact that $C_2(s)\to 0$ if $s\to 0$, we get an upper bound for $\{f^n\}$ uniformly in $n$:
$$
\|f^n\|_Q\leq \frac{\|f_I\|_{L^2}+C_0T + C_1 }{1-C_2(T)}
$$ 
if $T$ is suitably small such that $C_2(T)<1$. Then, by studying $f^{n+1}-f^n$ via 
$$
f^{n+1}-f^n= - \int_0^t \sigma'(\cdot)\int_Y b(y, z, \tau)(f^n-f^{n-1})dzd\tau, 
$$
we have 
$$
\|f^{n+1}-f^n\|_Q \leq C_2(T) \|f^n-f^{n-1}\|_Q.
$$
Thus one can conclude that $\{f^n\}_{n\in \mathbb{N}}$  is a Cauchy sequence in 
$Q$, which converges towards a solution $f$ of the equation (\ref{main}) for $C_2(T)<1$. 
Global existence for any $T$ then follows from a continuity 
argument by extending the local solution, proceeding as for the uniform estimate on $\{f^n\}_{n\in \mathbb{N}}$. 

\end{proof}

\subsection{Large time asymptotics, stability of steady states} \label{sec3.2} Clearly, it is important to study the forward dynamics of the residual neural network problem. Here we begin the discussion with the analysis of steady states and stability.

Note that the time horizon T is usually finite but often large in typical deep learning applications so that long time stability questions do become important, also for the design and implementation of appropriate space-time discretizations. Also it is important to realize that from a standpoint of network architecture design mildly stable evolutions are highly desirable while strongly stable ones lead to the cancellation or overdamping of fine structures in the network, which is a highly undesirable feature. Thus qualitative and quantitative stability information is needed, which can be used to impose additional constraints in the deep learning step, see e.g., \cite{HR17}.   

For simplicity, we first assume the forward propagation operator  to be autonomous. That is, 
$$
a=a(y), \quad b=b(y, z).
$$
At the first glance this looks like an oversimplification because the output parameter functions $a$ and $b$ of the deep learning optimization will - in realistic applications - be time-varying. However, it is clear that we cannot hope to understand the evolution dynamics of the non-autonomous problem without understanding the autonomous one. We shall show that even the latter gives rise to highly non-trivial and unexpected evolutive phenomena. Also, it is clear that long-time stability considerations of the autonomous problem are highly relevant for long time stability issues of the non-autonomous case (and for time-local stability of its time-discretisations). This will be detailed later on, after we have reached an understanding of the autonomous dynamics.

We make two basic assumptions here:\\
($A_1$) $\sigma$ is globally Lipschitz and \red{non-decreasing} on $\mathbb{R}$, $s\sigma(s) \geq 0$ on $\mathbb{R}$,  
but $\sigma'(0)>0$ and $\sigma(0)=0$,  \\
($A_2$) $a\in L^2(Y)$, $b\in L^2(Y\times Y)$.

Then the forward problem becomes 
\begin{subequations}\label{f1}
\begin{align}
 & f_t =\sigma(a-Bf),\\
 & f(y, 0)=f_I(y).
 \end{align}
 \end{subequations}
Here the operator defined by 
$$
(Bf)(y)=\int_Y b(y, z)f(z)dz
$$
from $L^2(Y)$ into $L^2(Y)$ is  Hilbert-Schmidt and consequently compact.
 
 Let $f_\infty \in L^2(Y)$ be a steady state, \red{i.e., $Bf_\infty =a$}. Note that steady states exist if and only if 
 $a\in R(B)$ ($=$ Range of $B$), and they are non-unique  if and only if $N(B)$ ($=$
 $\text{Nullspace}$ of $B$) is non-trivial. To study the stability of $f_\infty$, we first linearize 
 (\ref{f1}) at $f_\infty$, so that its perturbation $w$  satisfies 
 \begin{align*}
 & w_t =-\sigma'(0) Bw,\\
 & w(t=0)=w_I.
 \end{align*} 
 We obtain 
 $$
 w(t)=e^{-\sigma'(0)Bt}w_I.
 $$
 For an eigenvalue-eigenfunction pair  $(\omega, \phi)$ of $B$ we obviously have $e^{-\sigma'(0)\omega t} \phi$ as a solution of the linearized IVP.
 Therefore if the spectrum of $B$ contains an eigenvalue with negative real part, exponential instability holds for the linearized problem. 
 If an eigenvalue of $B$ with zero real part exists,  asymptotic stability for the linearized problem does not hold. 
 

To obtain stability for the linearized problem, we can impose  
$$
(Bv, v)_{L^2(Y)}=(B_sv, v)_{L^2(Y)} \geq 0 \quad \forall v\in L^2(Y), 
$$  
where $B_s=\frac{1}{2}(B+B^\top)$ is the symmetric part of $B$. 


Assume now that $B^\top=B$ and that all eigenvalues of $B$ are positive (i.e.,  $0$ is a spectral value but not an eigenvalue!).  Then asymptotic stability can be concluded from the solution representation 
  \begin{align*}
w(t)=\sum_{l=1}^\infty e^{-\sigma'(0) \omega_l t} (w_I, \phi_l)_{L^2(Y)}\phi_l,
 \end{align*} 
 where $\{\phi_l\}$ is the C.O.N.S (complete orthonormal system) of eigenfunctions. 
 
 If $\omega=0$ is an eigenvalue then stability (but not asymptotic stability) holds in the symmetric case. 
\red{\begin{rem}
Note that the above comments remain true in the non-autonomous case when  
$a_\infty:=\lim a(t)$, $b_\infty:=\lim b(t)$ exist (in an appropriate sense) and when $a, b$ are replaced by $a_\infty$ and,  respectively, $b_\infty$.  
\end{rem}
}
We now turn to discuss nonlinear stability for the forward propagation problem using Lyapunov functionals. 

Set
$$
u:=a-Bf,
$$
so that $u$ solves  
 \begin{align*}
 & u_t =- B\sigma(u),\\
 & u(t=0)=u_I:=a-Bf_I.
 \end{align*} 
In order to recover $f$ from $u$ we use the equation $f_t=\sigma(u)$ so that 
  \begin{align}\label{ff}
 f(y, t)=f_I(y)+\int_0^t \sigma(u(y, s))ds.
  \end{align}
 Multiply the u-equation by $\sigma(u)$ so that 
 $$
 \frac{d}{dt} \int_Y \Sigma(u)dy =-\int_Y (B_s \sigma(u), \sigma(u)) \leq 0, 
 $$
 (again assuming that $B_s \geq 0$).  This  gives after integration 
 $$
 \int_Y \Sigma(u(y, t))dy \leq  \int_Y \Sigma(u_I(y))dy,
 $$
 where 
 $$
 \Sigma(s)=\int_0^s \sigma(\xi)d\xi. 
 $$
 In order to obtain estimates for $f$ we distinguish two cases.  Firstly,   we assume that
 \begin{align}\label{ss+}
 \sigma(s)s \geq 0 \quad \text{for}\; s\not=0,
 \end{align}
 and  $|\sigma(s)|\geq C_1|s|$ for $|s|\geq C_2$, such that 
 $$
 \Sigma(s) \geq C_3 s^2
 $$
 for $|s|\geq C_4$. This allows to estimate $Bf$ the following way:
 \begin{align*}
 \int_Y|B f|^2dy \leq 2\int_Y |a|^2dy +2\int_Y |a-Bf|^2 dy \leq C \quad \forall t>0.   
 \end{align*}
 \red{Secondly, if} only $|\sigma(s)| \geq C_1$ for $|s|\geq C_2$, then
 $$
 \int_Y |Bf(t)| dy \leq C \quad \forall t>0
 $$
 follows.
 
 Similarly, from the equation $f_t=\sigma(u)$ with $u=a-Bf$ we conclude 
 $$
 \int_Y f_t Bf dy -\frac{d}{dt} \int_Y af dy =-\int_Y \sigma(u)u dy \leq 0,
 $$
 because of (\ref{ss+}).  Assume now that $B^\top=B$ and $B$ non-negative,  we then have 
 $$
 \frac{1}{2}\frac{d}{dt} \|B^{1/2}f\|^2_{L^2(Y)} -\frac{d}{dt}(a, f)_{L^2(Y)} =-\int_Y \sigma(u)u dy.
 $$ 
Thus
$$
\frac{1}{2} \|B^{1/2}f(t)\|^2_{L^2(Y)} - (a, f)_{L^2(Y)}
$$
is monotonically decreasing and bounded from below, admitting a limit as $t\to \infty$. 

Also 
$$
0\leq \int_0^\infty \int_Y \sigma(u(y, s))u(y, s)dyds <\infty
$$
and assuming $a\in R(B^{1/2})$, 
\begin{align*}
\|B^{1/2}f(t)\|^2_{L^2(Y)} & \leq K+2(a, f(t))_{L^2(Y)} \\
& \leq K+2(B^{-1/2}_s a, B^{1/2}_s f(t))_{L^2(Y)} \leq K+C\|B_s^{1/2}f(t)\|_{L^2(Y)}.
\end{align*}
Thus 
$$
\|B^{1/2}f(t)\|^2_{L^2(Y)} \leq K_1 \quad \forall t>0.
$$
Again, the projection of $f$ onto $N(B)$ is not controlled by this estimate. 

To collect facts, we have the following time-uniform estimates 
\begin{thm}
1) 
If  $B_s \geq0$, then for any $t>0$, we have \\
\begin{align*}
& (i)  \qquad \int_Y \Sigma(a-Bf(t))(y))dy \leq \int_Y \Sigma(a-Bf_I)(y))dy,  \\
\text{where} \; \; \Sigma' =\sigma, &  \; \text{and}  \\
& (ii) \qquad \int_0^\infty \|B_s^{1/2}\sigma(a-Bf(s))\|^2_{L^2(Y)}ds <\infty,  
\end{align*}
which implies that $B_s^{1/2}f_t:=B_s^{1/2}\sigma(u)\in L^2(Y\times (0, \infty))$. 
 \\
2) If $s\sigma(s) \geq 0$, $B^\top=B, \; B \geq 0$ and $a\in {R(B^{1/2})}$,  then for all $t>0$\\
\begin{align*}
& (iii)  \qquad  \|B^{1/2}f(t)\|_{L^2(Y)} +|(a, f(t))_{L^2(Y)}| \leq K. \\
& (iv) \qquad \int_0^\infty \int_Y \sigma(a-Bf(s))(y))\cdot (a-Bf(s))(y))dyds <\infty.
\end{align*}

\end{thm}
  \subsection{Characterization of steady states} 
 Note that  the equation 
 $$
  u_t=-B\sigma(u)
 $$
 may have other equilibria  than $u_e=0$. In fact every $u_e$ such that $\sigma(u_e)\in N(B)$ is an equilibrium. But $u_e=0$ is the only
one which may correspond to the equilibrium $f=f_\infty$ of the equation (\ref{f1}) (it does if only if $a\in R(B)$).
 Note that 
 $$
 u(t)=u_I -B\int_0^t \sigma(u(s))ds \Rightarrow   u(t)-u_I\in R(B).
 $$
 Since $u_I=a-Bf_I$ we have $u_I-a\in R(B)$ and $u(t)-a\in R(B)$. 
 Consider $0\not= u_e\in L^2(Y)$ such that $\sigma(u_e)\in N(B)$. Then the corresponding solution of the $f-$equation is 
 $$
 f(y, t)=f_I+t\sigma(u_e), \quad u_e=a-Bf_I, 
 $$  
 if $u_e-a\in R(B)$.  Clearly the linearly increasing component $t\sigma(u_e)\in N(B)$ is not seen by the time-uniform estimates of Theorem 3.2. 

To consider an example  pick $\phi\in L^2(Y)$, $\|\phi\|_{L^2}=1$. Define $u_e=\phi$, compute $\sigma_e=\sigma(\phi)$. Now choose 
$\psi \in \{\sigma_e\}^\bot$ and define the rank one operator 
$$
(Bf)(y):=\int_Y f(z)\psi(z)dz \phi(y).
$$ 
Clearly $u_e=\phi$ is an equilibrium of $ u_t=-B\sigma(u)$.
Also $u_e=\phi \in R(B)$. Now let $a=\alpha \phi \in R(B)$ ( $\alpha\in \mathbb{R}$ given) and choose $f_I\in L^2(Y)$  such that 
$$
\int_Y f_I \psi dy=\alpha-1.
$$
Then 
$$
f(t)=f_I +t\sigma(\phi)
$$
solves (\ref{f1}). 
\begin{lem}  If $B$ is symmetric, $\sigma(s)s \geq 0$ for all $s\in \mathbb{R}$ then every equilibrium $u_e$ in $R(B)$ satisfies $u_e\sigma(u_e)=0$. 
\end{lem}
\begin{proof}
Since $u_e$ is an equilibrium, $\sigma(u_e)\in N(B)$. The conclusion follows from $\overline{R(B)}=N(B)^\bot$, i.e.,
$$
\int_Y u_e \sigma(u_e)dy=0.
$$
Hence $u_e\sigma(u_e)\equiv 0$.
\end{proof}
We shall now analyze the stability of $u_e=0$ for the case of non-symmetric $B$. Therefore we consider the singular value decomposition of the operator $B$ \cite{Su}: 
$$
(Bf)(y)=\sum_{l=1}^\infty \mu_l(\psi_l, f)_{L^2(Y)}\phi_l(y),
$$  
where the singular values $\mu_l \geq 0$ are the eigenvalues of $|B|$, $\{\psi_l\}$ and $\{\phi_l\}$ are orthonomal systems, $\{\psi_l\}$ is complete 
in $L^2(Y)$  and $\phi_l =U\psi_l$, where $B=U|B|$ is the polar decomposition of $B$. Here $U$ is a partial isometry so that $N(U)=N(B)$. Set
$$
f=\sum_{k=1}^\infty f_k \psi_k, \quad f_k=\int_Y f\psi_kdy.
$$ 
Thus
\begin{align*}
(Bf, f)_{L^2(Y)} 
& =\sum_{l=1}^\infty\sum_{n=1}^\infty \mu_l f_l f_n (\phi_l,\psi_n)\\
&=f^\top DT f,
\end{align*}
where $f=(f_1,f_2,\cdots)^\top$, $D={\rm diag}(\mu_1,\mu_2,\cdots)$ and $T$ is the generalized Gram matrix: 
$$
T=((\phi_l, \psi_n)).
$$
Note that $(Bf, f)_{L^2(Y)}\geq 0$ for all $f\in L^2(Y)$ if and only if $DT$ is non-negative definite (not necessarily symmetric). 

Set 
$$
u(y, t)=\sum_{l=1}^\infty u_l(t)\phi_l(y)+Z(y, t),
$$
where $Z\in R(B)^\bot$. Thus 
$$
B\sigma(u(t))=\sum_{l=1}^\infty \mu_l(\sigma(u(t)), \psi_l)\phi_l\in R(B).
$$
Assuming again $a\in R(B)$ we have $u(\cdot, t)\in R(B)$ for all $t\geq 0$. We conclude $Z\equiv 0$ and 
$$
u(y, t)=\sum_{l=1}^\infty  u_l(t) \phi_l(y).
$$ 
Thus, we find 
\begin{subequations}\label{ss}
\begin{align}
& \frac{d}{dt} u_l=-\mu_l \int_Y \sigma \left( 
\sum_{k=1}^\infty u_k\phi_k(y)
\right) \psi_l(y)dy, \\
& u_l(t=0)=\int_Y u_I \phi_l dy.
\end{align}
\end{subequations}
Now let $B$ have rank $N<\infty$ (for simplicity's sake, in order to avoid the need to discuss the convergence of infinite series):
$$
(Bf)(y)=\sum_{l=1}^N \mu_l (f,\psi_l)_{L^2(Y)} \phi_l(y).
$$
\red{It is completely standard to show that  $u=0$ is an isolated equilibrium of (\ref{ss})
if the generalized $N\times N$ gram matrix $T=((\phi_k, \psi_l)_{L^2(Y)})$ is invertible. 
Also, it is easy to show that 
$$
L(u)=\int_Y \Sigma(u)dy
$$
is a Lyapunov function for (\ref{ss}) if the $N\times N$ matrix $DT$ with $D={\rm diag}(\mu_1,\cdots, \mu_N)$ is positive definite,  i.e., $L(u)>0$ and $L(u)$ decreases along trajectories around the origin. }

We conclude 

\begin{thm} Let $B$ have finite rank and let   \\
  (a) $T$ be invertible; \\
 (b) $DT$ be positive-definite (not necessarily symmetric).\\
 Then, the equilibrium $u=0$ is locally asymptotically stable. The convergence of $u(t)$ to zero is exponential. 
 \end{thm}
 We remark that  the local exponential stability of $u$ induces local exponential  stability of $f$.  The proof follows standard arguments using Lyapunov functionals for ODE systems \cite{Kh96}. 
  
 If $B$ is symmetric, then $\phi_l=\psi_l$, $T=id$ and $DT=D$ is positive definite if only if  $B > 0$.
 
 Also note that it is a simple exercise in functional analysis to do away with the finite rank assumption. The result carries over to the general case without change.

 \subsection{On solutions for the ReLu activation} 
 Note that for the arctan, sigmoid, and hyperbolic tangent activation functions, the asymptotic growth rate in time of the solution $f$ is at most linear, no matter what the properties of the operator $B$ are. In this respect, the ReLu and leaky ReLu activation functions behave worse as we shall show below.  
 
We now consider the  activation function $\sigma(s)=s^+$, which is one of the most popular activations used in practical applications. \red{In this case with $T>0$ assuming $a\in L^1((0, T); L^1(Y))$ and $b\in L^1((0, T); L^\infty(Y\times Y))$, from the equation for $f$  it follows by integration against sign$(f)$
 $$
 \frac{d}{dt}\int_Y|f(y, t)|dy \leq \int_Y |a(y, t)|dy +\sup_{Y\times Y}|b(t)|\int_Y |f(y, t)|dy.
 $$
 Thus, from the Gronwall inequality we find   
 $$
 \int_Y|f(y, t)|dy  \leq \left( \int_Y |f_I(y)|dy +\|a\|_{L^1((0, t); L^1(Y))}
 \right)\exp\left(\|b\|_{L^1((0, t); L^\infty(Y\times Y))}
  \right).
 $$
 for $0\leq t\leq T$.}  It is actually not difficult to construct an example of a rank 1 integral operator $B$ such that the exponential upper bound is sharp. This tells us that for $\sigma (s)=s^+$, exponential forward instability for $f$ is possible.

\red{In the autonomous case and} if $B_s \geq0$, we can actually prove that $\|f(\cdot, t)\|_{L^2(Y)}$ has at most linear growth in time.
\begin{prop} Let  $\sigma(s)=s^+$, $a=a(y), b=b(y, z)$, and $B_s=\frac{1}{2}(B+B^\top)\geq 0$. Then there exist $C_1, C_2>0$ such that 
$$
\|f(t)\|_{L^2(Y)}\leq C_1+C_2 t \quad \forall t>0.
$$
\end{prop}
\begin{proof}
From $f_t=u^+$ and $u=a-B f$ we have 
$$
u_t =-Bu^+.
$$
Using the Lyapunov argument from \S \ref{sec3.2} gives 
$$
\frac{1}{2}\frac{d}{dt} \int_Y (u^+(t))^2(y)dy=-(B_su^+, u^+)_{L^2(Y)}\leq 0.
$$
Thus 
$$
\int_Y (u^+(t))^2(y)dy \leq C \;\; \forall t>0. 
$$
This together with $f_t=u^+$ yields 
$$
\int_Y (f_t)^2 dy \leq C \; \;\forall t>0. 
$$
Using the relation $f(t)=f_I +\int_0^t \partial_s f ds$, we obtain the estimate as claimed.  
\end{proof}
Note that the same result holds for the leaky ReLu activation. 

\subsection{Local conditioning of the forward problem}\label{sec3.5}
For numerical analysis and computational purpose it is beneficial to understand  the conditioning of the forward propagation operator, 
which means that its linearization of the actual solution, not only the steady state must be looked at.   Also, the output of the learning problem will be time-dependent functions $a=a(y, t)$ and $b=b(y, z, t)$  such that the forward propagation and its linearization will be non-autonomous.  Consider a solution $u=u_f(y, t)$ of the forward problem (\ref{f1}) and analyze the linearization in direction $w=w(y, t)$,  with $u=a-B f$:
\begin{align}\label{f3}
\partial_t w(y, t)=-\sigma'(u(y, t))\int_Y b(y, z, t)w(z, t)dz. 
\end{align}
If the residual neural network problem is `very' deep, and if $u(t)$ is close to the stationary state $u\equiv 0$ (assuming that $a$ and $b$ 
stabilize sufficiently fast as $t\to \infty$), then the dynamics for $w$  will be close to the autonomous case considered above and controlled by the linearized autonomous problem in the beginning of section \ref{sec3.2}, when $b(y,z)$ is replaced by $b(y,z,t=\infty )$. This can be shown by standard semi-group perturbation theory, just to get a more quantitative feeling, multiply by $\frac{w}{\sigma'(u)}$ and integrate over $Y$:  
$$
\frac{1}{2} \frac{d}{dt}\int_Y \frac{w^2(y, t)}{\sigma'(u)}dy
=-(B(t)w, w)_{L^2(Y)} -\frac{1}{2}\int_Y \frac{\sigma''(u)u(t)}{(\sigma'(u))^2}w^2(t)dy.
$$
Here $-\frac{\sigma''(u)  u_t(t)}{(\sigma'(u))^2}$ measures the effect of the non-autonomous coefficients of the linearization at the local solution (instead of a stationary one).  If $f(t)$ is far away from the stationary state, then much less can be said about the operator $-\sigma'(w)B(t)$ in general, except that it is bounded by 
$$
\sup_{\mathbb{R}}|\sigma'|\|b(\cdot, \cdot, t)\|_{L^2(Y\times Y)}
$$ 
as an operator from $L^2(Y)$ into itself. It is generally non self-adjoint, even if $B(t)$ is self-adjoint.  

For the purpose of numerical discretization it is important to understand the time-local stability properties of the linearization \eqref{f3}. As usually done in ODE theory we freeze the coefficients of \eqref{f3} at a fixed time $t_0>0$ and find the problem
$$
\partial_t\tilde w(y,t) = -\int_Y \sigma'((y,t_0))b(y,z,t_0)\tilde w(z,t)dz,
$$
whose analysis is totally analogous to the first part of Section \ref{sec3.2} when $b(y,z)$ is replaced by $\sigma(u(y,t_0))b(y,z,t_0)$ and $\sigma'(0)$ by $1$. Globally growing/decaying modes $\tilde w$ represent local instability/stability of the nonlinear problem \eqref{f1} at time $t=t_0$ (see also \cite{He16a}).

Strong stability properties can be obtained by considering other conditions of neural networks. Here we just mention only one example (see \cite{RH18})
\begin{align*}
& f_t = B^\top \sigma(a-Bf),\\
& f(t=0) = f_I.
\end{align*}
A simple energy method shows that the associated flow in $L^2(Y)$ is contractive, if only $\sigma'\leq 0$ on $\mathbb{R}$, i.e., any two solutions $f_1,f_2$ satisfy
$$
||f_1(t)-f_2(t)||_{L^2(Y)}\leq ||f_1(t=0)-f_2(t=0)||_{L^2(Y)}.
$$

\section{Back propagation and optimal control} 
\subsection{Computing cost gradients}\label{sec4.1} The main technical difficulty in training continuous-depth networks is performing reverse-mode differentiation 
(also known as back propagation). We introduce the following notation: 
$$
a=a(y, t),\; b=b(y, z, t), \quad u_{a, b}:=a-B_bf, \quad f=f_{a, b},
$$
where $f_{a, b}$ solves (\ref{ft2}) below and 
$$
(B_bv)(y)=\int_Y b(y, z, t)v(z)dz.
$$
For the sake of simplicity in the calculation,  
we consider first optimizing a simple terminal value loss functional 
$$
J(a, b)=\frac{1}{2} \int_Y(f_{a, b}(y, T)- \tilde f(y))^2dy
$$
subject to 
\begin{subequations}\label{ft2}
\begin{align}
& \partial_t f_{a, b}=\sigma(a-B_b f_{a, b}),\\
& f_{a, b}(t=0)=f_I.
\end{align}
\end{subequations}
Here $\tilde f(y)$ is the target output function.  Let the Gateaux differential of $f$ in $a$ along direction $\alpha$ be 
$$
g=D_af_{a, b}(\alpha)=\lim_{\epsilon \to 0} \frac{1}{\epsilon}(f_{a+\epsilon \alpha, b}-f_{a, b}),
$$
then 
\begin{align*}
&g_t=\sigma'(u_{a, b})(\alpha-B_bg),\\
& g(t=0)=0.
\end{align*}
Let $M_{a, b}(t, s)$ be the evolution system \cite{Pa92} generated by $-\sigma'(u_{a, b})B_b$, i.e. $z(t):=M_{a, b}(t, s)z_0$ solves 
\begin{align*}
&  z_t=-\sigma'(u_{a, b}(t))B_bz, \; t\geq s,\\
& z(s)=z_0.
\end{align*}
Then 
$$
g(t)=\int_0^t M_{a, b}(t, s)(\sigma'(u_{a, b}(s)\alpha(s))ds.
$$
Similarly, $h=D_bf(\beta)$ solves 
\begin{align*}
& h_t=-\sigma'(u_{a, b})(B_bh +B_\beta f), \\
& h(t=0)=0,
\end{align*}
which gives 
$$
h(t)=-\int_0^t M_{a, b}(t, s)(\sigma'(u_{a, b}(s)B_\beta f(s))ds.
$$
We proceed to compute the derivatives of $J$ with respect to $a$ and $b$ as follows: 
\begin{align*}
D_aJ(a, b)(\alpha) & =\int_Y (f_{a, b}(y, T)-\tilde f(y))g(y, T)dy \\
& = \int_Y (f_{a, b}(y, T)-\tilde f(y))\int_0^T M_{a, b}(T, s)(\sigma'(u_{a, b}(s)\alpha(s))(y) dsdy \\
&=\int_0^T\int_Y \alpha(y, s) \sigma'(u_{a, b}(y, s))M_{a, b}(T, s)^*(f_{a, b}(y, T)-\tilde f(y))dyds.
\end{align*}
Thus 
\begin{align*}
D_aJ(a, b)(y, s) =\sigma'(u_{a, b}(y, s))M_{a, b}(T, s)^*(f_{a, b}(T)-\tilde f)(y).
\end{align*}
Define 
$$
r_T(y):=(f_{a, b}(T)-\tilde f)(y).
$$
Clearly, $r(s):=M_{a, b}(T, s)^*r_T$ solves the co-state terminal value problem, 
\begin{align*}
&  r_t=(\sigma'(u_{a, b}(s)B_b)^*r= B_b^*(\sigma'(u_{a, b})(s)r),\\
& r(T)= r_T.
\end{align*}
Note that $B_b^*=B_{b^\top}$ with $b^\top(y, z, t)=b(z, y, t)$. Thus, $r(s)=r_{a, b}$, and $r_{a, b}$ solves 
\begin{subequations}\label{co}
\begin{align}
& \partial_t r_{a, b}= B_{b^\top}(\sigma'(u_{a, b})(s)r_{a, b}(s)),\\
& r_{a, b}(T)= f_{a, b}(T)-\tilde f.
\end{align}
\end{subequations}
We conclude  
\begin{align}\label{daj}
D_aJ(a, b)(y, s)=\sigma'(u_{a, b}(y, s))r_{a, b}(y, s).
\end{align}
As for the gradient with respect to $b$ we have 
\begin{align*}
& D_bJ(a, b)(\beta) \\
& =\int_Y (f_{a,b}(y, T)-\tilde f(y))h(y, T)dy \\
& = -\int_Y (f_{a, b}(y, T)-\tilde f(y))\int_0^T M_{a, b}(T, s)(\sigma'(u_{a, b})(s)B_\beta f_{a,b}(s))(y)dsdy \\
& =- \int_0^T\int_Y \sigma'(u_{a, b}(s)) B_\beta f_{a, b}(s)M_{a, b}(T, s)^*(f_{a, b}(T)-\tilde f)(y)dyds\\
&=-\int_0^T\int_Y\int_Y \beta(y,z, s)f_{a, b}(z,s) \sigma'(u_{a, b}(y, s))M_{a, b}(T, s)^*(f_{a, b}(T)-\tilde f)(y)dydzds\\
&=-\int_Y\int_Y \int_0^T\beta(y, z, s) f_{a,b}(z,s)\sigma'(u_{a, b}(y, s))r_{a,b}(y,s)ds dydz.
\end{align*}
This gives 
\begin{align}\label{dbj}
D_bJ(a, b)(y, z, s)=-f_{a,b}(z,s)\sigma'(u_{a, b}(y, s))r_{a, b}(y, s).
\end{align}
We collect the results on the gradient of $J$ in the following:
\begin{prop} \label{prop4.1} 
We have 
\begin{align*} 
(i) \qquad & D_aJ(a, b)(y, s)=\sigma'(u_{a, b}(y, s))r_{a, b}(y, s), \\
(ii) \qquad & D_bJ(a, b)(y, z, s)=-f_{a,b}(z,s)\sigma'(u_{a, b}(y, s))r_{a, b}(y, s).
\end{align*}
\end{prop}
Therefore, conditions (necessary and sufficient) for a stationary point 
$$
(a, b)\in L^2(Y \times (0, T))\times L^2(Y\times Y \times (0, T))
$$
of the functional $J(a, b)$ are:  \\
(a)  solve 
$$
f_t=\sigma(a-B_bf),0<t \leq T, \quad f(t=0)=f_I
$$
for $f=f_{a, b}=f_{a, b}(y,t)$, $u_{a, b}:=a-B_bf_{a,b}$; \\
(b) solve 
\begin{align*}
& r_s=B_{b^\top}(\sigma'(u_{a, b})(s)r(s)),\; 0\leq  s < T,\\
& r(s=T)= f_{a,b}(T)-\tilde f
\end{align*}
for $r=r_{a,b}=r_{a,b}(y,s)$. Then the stationarity condition is 
\begin{align}
\sigma'(u_{a,b}(y,s))r_{a,b}(y,s)=0, \quad a.e. \; y\in Y, \quad s\in (0,T);
\end{align}
while (ii) of Proposition \ref{prop4.1} does not add any additional information.
\begin{rem}\label{rem4.1}
If $\sigma'>0$ (this holds for arctan, hyperbolic tangent, and Sigmoid), the above condition implies that the optimal $(a^*,b^*)$ exists if only if $\tilde f$ is reachable in the sense that the above two derivatives vanish if and only if there exist parameter functions $a^*$ and $b^*$ such that $f_{a^*,b^*}(y,T)=\tilde f(y)$ a.e. in $Y$. For the ReLu activation instead the situation is entirely different as any $f_{a,b}$ with $a\leq B_bf_{a,b}\in Y\times (0,T)$ is a stationary point of $I(a,b)$, which makes the search of a minimizer in general very difficult.
\end{rem}

\begin{rem}\label{rem4.2}
Note that the conclusion of Remark \ref{rem4.1} does not hold if the cost functional is regularized by, say,  
the Tikhonov regularizer 
\begin{align*}
R(a,b) & =\int_{Y} \int_0^T |a(y,t)|^2 dt dy +\int_{Y\times Y} \int_0^T|b(y,z, t)|^2dt dydz \\
\end{align*}
such that $J(a, b)$ is replaced by 
\begin{align}\label{jmod}
J_{\rm mod}(a,b):=J(a, b) +\lambda R(a,b).
\end{align}
Then 
\begin{subequations}\label{Dm}
\begin{align}
& D_a J_{\rm mod}(a,b)=(\sigma'(u_{a, b})r_{a, b})(y, s) +\lambda a(y, s),\\
&D_b J_{\rm mod}(a,b)= - f_{a, b}(z, s)(\sigma'(u_{a, b})r_{a, b})(y, s) +\lambda b(y, z, s)
\end{align}
\end{subequations} 
This is commonly done in the deep learning applications.  
\end{rem}
The above analysis is well generalizable to the classification problem of section \ref{sec2.2} and section \ref{sec2.3} for which only the final cost needs to be modified by 
$$
J(a, b)=\frac{1}{2}\int_Y |C^{\rm pre}(y)-C(y)|^2 dy
$$
with 
$$
C^{\rm pre}(y)=h(O_T(y)), \quad O_T(y)=\int_Y W(y, z)f(z, T)dz+\mu(y). 
$$
For the back propagation, we obtain the same equation  
\begin{align*}
 r_s=B_{b^\top}(\sigma'(u_{a, b})(s)r(s)),\; 0\leq  s < T,
\end{align*}
but with a different terminal condition
$$
r(z, T)=\int_Y (C^{\rm pre}(y)-C(y))h'(O_T(y))W(y, z)dy.
$$

\subsection{Pontryagin Maximum Principle} \label{sec4.2}
We now view the deep learning problem in the framework of the mathematical control theory using the Pontryagin maximum principle to obtain optimal controls for the network parameter functions $a$ and $b$, see \cite{FR75}. This is a standard application of control theory but is of particular interest, when optimizers $(a, b)$ which vary in regions with boundaries, are sought, possibly as an alternative to the Tikhonov functional regularization \eqref{jmod}.  

\red{In fact for the deep learning problem the maximum principle can be used to (rather) explicitly compute the optimal parameter function in terms of the optimal state and co-state variables {obtaining bang-bang} type controls. }

Let $a=a(y, t)$ and $b=b(y, z, t)$ be  in a measurable set $A\subset \mathbb{R}^2$ pointwise a.e., now (looking for maximizers instead of minimizers, to keep in line with the usual convention in control theory). Define 
$$
I(a, b)=- \frac{1}{2} \int_Y (f_{a, b}(y, T)-\tilde f(y))^2dy.
$$
Look for 
$$
\max_{(a,b)\in A}I(a, b)=I(a^*, b^*).
$$
As usual in control theory we define the Hamiltonian
$$
H(f, r, a, b):=\int_Y \sigma(a-B_bf)r dy,
$$
where $r$ is the co-state variable.    Let $(a^*, b^*)$ be optimal for $I$. Define $f^*=f_{a^*,b^*}$, then 
\begin{align*}
&  f_t^*=\sigma(a^*-B_{b^*}f^*),\; 0\leq t\leq T,\\
& f^*(t=0)=f_I.
\end{align*}
Also define the optimal co-state $r^*$ by 
\begin{align*}
&  r_t^*=   B_{(b^*)^\top}(\sigma'(a^*-B_{b^*}f^*)r^*),\; 0\leq t\leq T, \\
& r^*(t=T)=\tilde f -f^*(T),
\end{align*}
where $(b^*)^\top(y, z, t)=b^*(z, y, t)$.
Then $(a^*, b^*)$ satisfies the Pontryagin maximum--principle. 
\begin{align}\notag 
H(f^*, r^*, a^*, b^*)& =\max_{(a, b)\in A} H(f^*, r^*, a, b)\\
& =\max_{(a, b)\in A}\int_Y\sigma(a-B_b f^*)r^*dy.
\end{align} 
It is a classical result of control theory and we refer the reader to \cite{BCD97}. Note that the Hamiltonian is constant along the coupled dynamics: 
$$
\frac{d}{dt} H(f^*(t), r^*(t), a^*(t), b^*(t))=0.
$$
As an example, take $A=[a^-, a^+]\times [b^-, b^+] \subset \mathbb{R}^2$. 
Let $\sigma$ be strictly increasing in $\mathbb{R}$. Then one concludes immediately defining $\chi_\Omega$ as the indicator function on the set $\Omega$,  
\begin{align*}
a^*(y, t)\chi_{\{ r^*\neq 0\}} & =a^+\chi_{\{r^* > 0\}}(y, t) +a^-\chi_{\{r^*< 0\}}(y, t),\\
 \text{and} & \\
& b^*(y,z, t)\chi_{\{ r^*\neq 0\}}(y,t)\chi_{\{ f^*\neq 0 \}}(z,t)\\ 
& =b^-\left(\chi_{\{r^*> 0\}}(y, t)\chi_{\{f^*\geq 0\}}(z, t) + \chi_{\{r^*< 0\}}(y, t)\chi_{\{f^*<0\}}(z, t)\right)\\
& \qquad +b^+\left(\chi_{\{r^*> 0\}}(y, t)\chi_{\{f^*<0\}}(z, t) + \chi_{\{r^*< 0\}}(y, t)\chi_{\{f^*>0\}}(z, t)\right).
\end{align*} 
Note that this defines $a^*$ and $b^*$ completely iff the sets where $r^*$ and $f^*$ equal $0$ have both zero Lebesgue measure in $Y\times (0,T)$. For a general control set $A$, compact in $\mathbb{R}^2$ , set:
$$
K_A=\left\{(a, b)\in  \mathbb{R} \times L^2(Y)\Big|\quad (a, b(z))\in A\; a.e. \;\text{in}\;\; Y \right\}.
$$
$K_A$ is closed in $ \mathbb{R} \times L^2(Y)$. For $f\in L^2(Y)$ define the affine linear functional 
\begin{align}\label{tf}
T_f(a, b)=a-\int_Y b(z)f(z)dz.
\end{align} 
Clearly, $T_f: K_A\to \mathbb{R}$ assumes its minimum at  $(a_f^-, b_f^-)$ and maximum at $(a_f^+, b_f^+)$ in $K_A$ since 
$T_f$ is bounded on $K_A$, weakly  continuous and minimizing and maximizing sequences in $K_A$ have  weakly converging subsequences 
in $K_A$. Clearly, uniqueness of argmin ad argmax does not hold in general (for example if the Lebesgue measure of the set where $f=0$ is positive). Again assuming that $\sigma$ is strictly increasing, there are pairs in the set of argmin and argmax such that:
\begin{subequations}\label{49}
\begin{align}
a^*(y, t)\chi_{\{r^*\geq 0\}} & =a^+_{f^*(t)}\chi_{\{r^*> 0\}}(y, t) +{a}^-_{f^*(t)}\chi_{\{r^*< 0\}}(y, t),\\
b^*(y,z, t)\chi_{\{r^*\geq 0\}} & =b^+_{f^*(t)}(z)\chi_{\{r^*> 0\}}(y, t)+b^-_{f^*(t)}(z)\chi_{\{r^*<0\}}(y, t).
\end{align} 
\end{subequations}
Similarly as above this does not define $a^*$ and $b^*$ in full generality.

Note that the forward evolution for $f^*$ and the backward evolution for the co-state $r^*$ are now coupled in a highly nonlinear way through the optimal controls $(a^*, b^*)$. Existence and uniqueness issues for this highly nonlinear initial-terminal value problem will be the subject of future work.

In particular we remark that the Pontryagin Maximum Principle does not give any information on optimality if the state $\tilde f$ is reachable by a control in $K_A$. In this case $r^*=0, max I=0$ and the optimal control has to be computed as in Section \ref{sec4.1}.

\begin{rem} For the network loss function of the classification problem 
$$
I(a, b)=-\frac{1}{2}\int_Y|C_{a, b}^{\rm pre} -C|^2dy,
$$
with 
$$
C_{a, b}^{\rm pre}(y)=h\left(\int_Y f_{a, b}(z, T)W(z, y)dz+\mu(y)\right)=h(O_{a, b, T}(y)), 
$$
the only modification again is 
 the  terminal value of the co-state, 
$$
r^*(T)= - \int_Y (C_{a^*, b^*}^{\rm pre}(y)-C(y))h'(O_{a^*, b^*, T}(y))W(y, z)dy.
$$
\end{rem}

\subsection{Functional Halmilton-Jacobi-Bellman PDE}
We now present an alternative approach to the control problem based on the dynamic programming principle.
Consider 
\begin{align*}
\partial_s f(y,s) & = \sigma(a(y,s)-(B_b f)(y,s)), \quad t<s \leq T,\\
 f(y, t) & = v(y)
\end{align*} 
for general $v(\cdot)\in L^2(Y)$.  Let  a general cost functional be defined by 
$$
J_{v, t}(a, b)=\int_t^T\int_Y L(f(y, s), a, b)dyds +\frac{1}{2}\int_Y (f(y, T)-\tilde f)^2dy,
$$
where the first term denotes the running cost and the second term is a terminal cost. Define a value functional as 
$$
F(v, t)=\inf_{(a, b)\in A} J_{v, t}(a, b)=J_{v, t}(a^*,b^*).
$$
Note that $F(v, T)=\frac{1}{2}\int_Y (v(y)-\tilde f)^2dy$.  By the dynamic programming principle (see e.g.,\cite{BCD97})
we conclude 
\begin{thm}\label{HJB}
Assume the value functional $F$ is smooth in its arguments $(v, t)$. Then $F(v, t)$ solves the functional Hamilton-Jacobi-Bellman (HJB) 
equation 
\begin{align}\label{Ft}
 \partial_t F(v, t)+\min_{(a, b)\in A} \left\{ \int_Y D_vF(v, t) \sigma(a-B_b v)dy +\int_Y L(v, a, b)dy\right\} = 0
\end{align}
with the terminal condition 
\begin{align}\label{FT}
F(v, T)=\frac{1}{2}\int_Y (v(y)-\tilde f(y))^2dy.
\end{align}
\end{thm}
\begin{rem} Note that $D_vF(v, t)$ is the $L^2$ variational gradient of the functional $F(t): L^2(Y) \to \mathbb{R}$. \\
(i) We can express the HJB as
$$
\partial_t F(v, t)+H(v, D_v F(v, t))=0, 
$$
where we define  the Hamiltonian as
$$
H(v, r)=\min_{(a, b)\in A} \left\{ \int_Y \sigma(a-B_b v) r dy +\int_Y L(v, a, b)dy\right\}.
$$ 
It is easy to see that the characteristic system of this functional HJB equation in the case $L=0$ is precisely the coupled optimal control system of the previous section. 

Note that the HJB equation `lives' in the space of functionals on the space $L^2(Y)$.
\end{rem}
\red{Theorem  \ref{HJB} is an important statement that links smooth solutions of the HJB equation with solutions of the optimal control problem, and hence the minimization problem (\ref{op}) in deep learning.
By taking the $\min$ in (\ref{Ft}), the HJB allows to identify the optimal control $(a, b)$.  In this sense, the HJB equation gives a necessary and sufficient condition for optimality of the learning problem (\ref{op}).This demonstrates an essential observation from the optimal control viewpoint of deep learning: the minimization can be viewed as a variational problem, whose solution can be characterized by a suitably defined Hamilton-Jacobi-Bellman equation. This very much parallels classical calculus of variations.  
However, we should note there is a price to pay for obtaining such a feedback control: the HJB equation is general difficult to solve numerically. }

Nevertheless, we present main steps for designing the optimal control $(a^*, b^*)$ using the above dynamic programming approach.\\ 

\noindent{\bf Step 1}. Solve the HJB equation 
$$
\partial_t F(v, t)+H(v, D_v F(v, t))=0  \quad 0\leq t\leq T, 
$$
subject to the terminal condition (\ref{FT}) to find the value functional $F(v, t)$. \\

\noindent{\bf Step 2}. Use $F(v, t)$ and the HJB equation to construct an optimal $(a^*, b^*)$:\\
(i) for each $v \in L^2(Y)$ and each time $t\in [0, T]$, define 
$$
(\tilde a(v(t))(y), \tilde b(v(t))(y, z))={\rm argmin}_{(a, b)\in A} \left\{\int_Y D_v F(v, t) \sigma(a-B_b v)dy +\int_Y L(v, a, b)dy \right\}.
$$ 
(ii) Next we find $\tilde f(y, s)$ by solving the following PDE 
\begin{align*}
& \partial_s \tilde f=\sigma( \tilde a(v)(y, t)-B_{ \tilde b(v)(y, z, t)} \tilde f), \quad t\leq s \leq T, \\ 
& \tilde f(t)=v.
\end{align*}
(iii) Finally define the feedback control 
$$
a^*(y, s):=\tilde a(\tilde f(s))(y),\; b^*(y, z, s):=b(\tilde f(s))(y, z),  \quad t \leq s\leq T.
$$
 \begin{thm}
 The control $(a^*, b^*)$ is optimal.
 \end{thm}
 \begin{proof}
 By standard arguments from dynamic programing, see \cite{BCD97}.
 \end{proof}
\red{It is worth noting that the HJB equation is a global characterization of the value function, in the sense that it must in principle be solved over the entire space of input-target distributions. Of course, one would not expect this to be the case in practice for any non-trivial machine learning problem;  hence it would be desirable to solve the HJB locally by some Lagrangian approach in order to apply to nearby input-label samplings.  Another limitation of the HJB formulation is that it assumes the value function is smooth, which is often not the case. A more flexible characterization of the value function is  to relax the solution space in an appropriate sense, such as the viscosity sense \cite{BC97}.  }

\section{Two iterative algorithms} 
\red{ 
Deep Neural Networks have drastically advanced the state-of-the-art performance in many computer science applications, yet in the face of such significant developments, the age-old stochastic gradient descent (SGD) algorithm \cite{RM51} remains one of the most popular method for training DNNs. Finding new and simple hyper-parameter tuning routines that boost the performance of state of the art algorithms remains one of the most pressing problems in machine learning (see, e.g., \cite{BCN18, GBC16}). Based on the gradients obtained in Section \ref{sec4.1}, and the Pontryagin maximum principle presented in Section \ref{sec4.2}, we will allude briefly to two respective algorithms  in this section.
 }
 \subsection{Gradient descent} We recall that the gradient of the cost functional   
 $$
 J(a, b)=\frac{1}{2}\int_Y (f_{a, b}(y, T)-\tilde f(y))^2dy
 $$
 is given by
 $$
 D_aJ=\sigma(u_{a, b}(y, s))r_{a, b}(y, s), \quad D_bJ= -f_{a, b}(z, s)\sigma(u_{a, b}(y, s))r_{a, b}(y, s),
 $$  
 where $u_{a, b}=a-B_b f_{a, b}$, and $r_{a, b}$ is obtained by solving 
 \begin{align*}
 & \partial_t r_{a, b}=B_{b^\top}(\sigma'(u_{a, b}(y, s))r_{a, b}(y,s),\quad 0\leq s \leq T,\\
 & r_{a, b}(\cdot, T)=f_{a, b}(\cdot, T)-\tilde f(\cdot).
 \end{align*}
We remark that the conditioning of the backward problem for $r_{a, b}$ is identical to the conditioning of the forward problem. 
 More precisely, with $\tau=T-s$, $0\leq \tau \leq T$, $R_{a, b}(\tau):=r_{a, b}(s)$ we obtain 
 $$
  \partial_tR_{a, b}(y, \tau)=- B_{b^\top(T-\tau)}(\sigma'(u_{a, b}(y, T-\tau)))R_{a, b}(y,\tau),\quad 0\leq \tau \leq T.
 $$
 Note that the generator of the evolution equation for $R_{a, b}$ at $\tau$ is precisely the transposed of the generator of the linearized 
 convolution equation for $f_{a, b}$ at time $T-\tau$  and we can estimate 
$$
\|B_{b^\top(T-\tau)}\circ \sigma'(u_{a, b}(T-\tau))\|_{L^2(Y)\to L^2(Y)} \leq 
\sup_{\mathbb{R}}|\sigma'| \|b(\cdot, \cdot, T-\tau)\|_{L^2(Y\times Y)}.
$$
(compare to section \ref{sec3.5}). 

 \red{GD and SGD have advantages of easy implementation and being fast for well-conditioned and strongly convex objectives. However, they have convergence issues, especially when the problem is ill-conditioned;  there is an extensive volume of research for designing algorithms to speed up the convergence (see, e.g., \cite{DHS11,KB15,Po64,TH12}).  To achieve fast convergence with large time steps (learning rates) we present the following algorithm.}
 \\[0.5cm]    
 \noindent{\bf Algorithm 1.}\\  
 {\bf Inputs}:  $\tilde f(y)$, $f_I(y)$, $a^0, b^0$ as initial guess, step size $\tau$.\\
  {\bf Outputs}:  $a, b$ and $J(a, b)$ \\
  1. For $k=1, 2, \cdots$ iterate until convergence. \\
  2. Employ the Proximal Alternating Minimization (PAM) method \cite{ABRS10} for $a$ and $b$,   
  \begin{subequations}\label{abk}
  \begin{align}
  a^{k+1}& ={\rm argmin}_a \left\{J(a, b^k)+\frac{1}{2\tau}\|a-a^k\|^2 \right\}.\\
  b^{k+1}& ={\rm argmin}_b \left\{J(a^{k+1}, b)+\frac{1}{2\tau}\|b-b^k\|^2 \right\}.
    \end{align}
      \end{subequations}
  3. Update $f$ as 
  $$
  f^{k+1}=f_{a^{k+1}, b^{k+1}}(y, s)
  $$
  by solving 
  $$
  \partial_sf=\sigma(a^{k+1}-B_{b^{k+1}} f),\quad f(t=0)=f_I.
  $$
Note that this algorithm needs to be modified when the cost functional is regularized.  For the Tikhonov regularizer 
given in Remark \ref{rem4.2},  we replace $J(a, b)$ by $J_{\rm mod}(a, b)$ defined in (\ref{jmod}) and use (\ref{Dm}) for the gradients.  

\red{ For a class of objective functions, (\ref{abk}) is analyzed in \cite{ABRS10}, where it is called the Proximal Alternating Minimization (PAM) method. Here, at each step, the distance of the parameter update acts as a regularization to the original loss function. Compared to GD (or SGD), the PAM has the advantage of being monotonically decreasing, which is guaranteed for any step
size $\tau>0$.}  Indeed, by the definition of $(a^{k+1}, b^{k+1})$ in (\ref{abk}), 
   \begin{align*}
 J(a^{k+1}, b^{k+1}) \leq J(a^{k}, b^k)-\frac{1}{2\tau}\left(\|a^{k+1}-a^k\|^2 +\|b^{k+1}-b^k\|^2\right).
  \end{align*}
We remark that (\ref{abk}a) is the celebrated proximal point algorithm (PPA) \cite{Ro76}.  
PPA based implicit gradient descent algorithms have been explored in \cite{YPBO18} for the classic $k$-means problem, and in \cite{Cet16}  for accelerating the training of DNNS. 

    
 We should point out that training deep neural networks using gradient-based optimization fall into the noncovex nonsmooth optimization. Many researchers have been working on mathematically understanding the GD method and its ability to solve nonconvex nonsmooth problems (see, e.g., \cite{ACGH18, Ka16, Ne13, SS18}). Accelerating the gradient method is also a subject of intensive studies (see, e.g., \cite{SBJ15, WWJ16}).
 
 \subsection{Hamiltonian maximization}
 \red{When training is recast as a control problem, necessary optimality conditions are formulated by the Pontryagin maximum principle (PMP). This formulation can lead to an alternative framework for training algorithms. There are actually many methods for the numerical solution of the PMP (see the survey article \cite{Ra09}), here we follow the method of successive approximations (MSA) \cite{CL82}, which is an  iterative  method  based  on  alternating  propagation  and  optimization steps. For recent works using PMP based MSA algorithms to train neural networks, we refer to \cite{LCTE18, LH18}. 
  }
  
Recall the Hamiltonian of the form 
 $$
 H(v, r, a, b)=\int_Y \sigma(a-B_b v) r dy.
 $$
We thus present the following algorithm. 
\\[0.5cm]
  \noindent{\bf Algorithm 2.}\\  
 {\bf Inputs}:  $\tilde f(y)$, $f_I(\cdot)$, $a^0, b^0$ as initial guess.\\
  {\bf Outputs}:  $a, b$ and $J(a, b)$ \\
  1. For $k=1, 2, \cdots$ iterate until convergence. \\
  2. find $f^k=f_{a^k, b^k}$ by solving the forward problem 
  $$
 \partial_t f=\sigma(a^k-B_{b^k}f), \quad f(t=0)=f_I.
  $$
  3. find $r^k=r_{a^k, b^k}$ by solving the backward problem 
  $$
 \partial_t r =B_{b^\top} (\sigma'(a^k-B_{b^k}f^k),\quad r(t=T)=\tilde f -  f^k(T).
  $$
  4. Update $(a, b)$ by 
  $$
(a^{k+1}, b^{k+1})={\rm argmax}_{(a, b)\in A} H(f^k, r^k, a, b).
  $$
 Since $\sigma$ is non-decreasing,  the linear programing problem (\ref{tf}) (see Section \ref{sec4.2}) may be used to update $(a, b)$.  
 
\red{As is the case with the maximum principle,  the above algorithm consists of two major components: the forward-backward Hamiltonian dynamics and the maximization for the optimal parameters at each step.  An important feature of the algorithm is that the Hamiltonian maximization  is  decoupled  for  each step.  In  the  language  of  deep  learning,  the  optimization step is decoupled for different network layers and only the Hamiltonian involves propagation through the layers.  This allows the parallelization of the maximization step, which is typically most time-consuming.}
 
One advantage of this approach is that it does not rely on gradients with respect to the trainable parameters  through  back-propagation. \red{An additional advantage is that one has a good control of the error through explicit estimates on the Hamiltonian (see \cite{LCTE18}). Overall, the approach opens up new avenues to attack training problems associated with deep learning. }
  
Finally, we point out that both the forward and backward PDE problems when discretized by numerical methods can lead to different network architectures (with respect to depth and width). \red{Implementation and convergence analysis of the above two learning algorithms with proper discrete network architectures for specific application tasks are left  to further work.}


\section*{Acknowledgement}  We are grateful to Michael Herty (RWTH) for his interest, which motivated us to investigate this problem and eventually led to this paper. Liu was partially supported by the National Science Foundation under Grant DMS1812666 and by NSF Grant RNMS (Ki-Net)1107291.

\bibliographystyle{abbrv}

\end{document}